\documentclass[12pt]{amsart}
\usepackage[totalwidth=510pt, totalheight=640pt]{geometry}
\usepackage{amsmath, amssymb, amsthm, amsfonts, 
	eucal, enumerate, natbib, layout}
\usepackage[normalem]{ulem}
\usepackage[all,tips,2cell]{xy}
\usepackage[usenames]{color}
\usepackage{verbatim}
\usepackage{soul}
\usepackage{mathtools}
\usepackage{hyperref}
\hypersetup{colorlinks=true,citecolor={red},linkcolor={blue},urlcolor={violet}}
\usepackage{graphicx}

\newcommand{\details}[1]{}

\newcommand{\hookdownarrow}{\mathrel{\rotatebox[origin=c]{-90}{$\hookrightarrow$}}}

\newtheorem{theorem}{Theorem}[section]
\newtheorem*{theorem*}{Theorem}

\newtheorem*{corollary*}{Corollary}

\newtheorem*{lemma*}{Lemma}

\newtheorem*{claim*}{Claim}

\newtheorem{proposition}[theorem]{Proposition}
\newtheorem*{proposition*}{Proposition}

\newtheorem*{conjecture*}{Conjecture}
\newtheorem{def-proposition}[theorem]{Definition-Proposition}
\theoremstyle{definition}
\newtheorem{definition}[theorem]{Definition}
\newtheorem*{definition*}{Definition}
\newtheorem{remark}[theorem]{Remark}

\newtheorem*{example*}{Example}
\numberwithin{equation}{section}

\newcommand{\pn}{\noindent}

\newcommand{\ZZ}{\mathbb{Z}}
\newcommand{\QQ}{\mathbb{Q}}
\newcommand{\RR}{\mathbb{R}}
\newcommand{\CC}{\mathbb{C}}
\newcommand{\GG}{\mathbb{G}}

\newcommand{\Hom}{\mathrm{Hom}}
\newcommand{\longhookrightarrow}{\lhook\joinrel\longrightarrow}

\newcommand{\MT}{\mathrm{MT}}
\newcommand{\W}{\mathrm{W}}
\newcommand{\F}{\mathrm{F}}
\newcommand{\T}{\mathrm{T}}
\newcommand{\HH}{\mathrm{H}}
\newcommand{\R}{\mathrm{R}}

\newcommand{\im}{\mathrm{Im}}

\newcommand{\ii}{\mathrm{i}}
\newcommand{\Gr}{\mathrm{Gr}}

\newcommand{\li}{\mathrm{Lie}\,}

\newcommand{\Ext}{\mathrm{Ext}}
\newcommand{\uHom}{\underline{\mathrm{Hom}}}
\newcommand{\uExt}{\underline{\mathrm{Ext}}}

\newcommand{\V}{\mathbf{V}}
\newcommand{\h}{\mathrm{H}}

\newcommand{\cO}{\mathcal{O}}

\newcommand{\ad}{\mathrm{ad}}
\newcommand{\an}{\mathrm{an}}

\newcommand{\locsys}{\mathrm{Loc.Sys}(S)}
\newcommand{\spec}{{\mathrm{Spec}}\,}

\newcommand{\vmhs}{\mathrm{VMHS}}

\newcommand{\mot}{\mathrm{Mot}}

\begin{document}

\title[1-motives and admissible variations of mixed Hodge structures]
{1-motives and admissible variations of mixed Hodge structures}

\author{Cristiana Bertolin}
\address{Dipartimento di Matematica, Universit\`a di Padova, Via Trieste 63, Padova}
\email{cristiana.bertolin@unipd.it}

\subjclass[2020]{14D07, 14C30, 14C15, 14K30} 

\keywords{1-motives over a scheme, admissible variations of mixed Hodge structures, the enriched Hodge realization}

\date{\today}


\begin{abstract}

	Let $S$ be a connected scheme smooth and of finite type over $\CC$. To every 1-motive over $S$, Andr\'e associated the enriched Hodge realization given by a torsion-free, graded-polarizable and admissible variation of mixed Hodge structures of type $(0,0)$, $(-1,0)$, $(0,-1)$, $(-1,-1)$ over the associated complex analytic space $S^\an$. In this paper, we prove that every admissible variation of mixed Hodge structures of the above type arises, up to isogeny, from a 1-motive over $S$, thereby providing a positive answer to a question of Andr\'e \cite[Question 3.2.3]{A25} concerning the geometric origin of such variations. More precisely, we establish a Hodge-theoretic interpretation of sections of semi-abelian varieties by combining Andr\'e's description of the abelian case with a new analysis of the toric part. 
	
	As a consequence, we prove a relative analogue of Deligne’s equivalence over \(\mathbb C\) \cite[Construction (10.1.3)]{D1}. Namely, under suitable assumptions on \(S\) and on the lattices and the tori underlying 1-motives, the enriched Hodge realization functor induces an equivalence between the category of \(1\)-motives over \(S\) and the category of torsion-free, graded-polarizable and admissible variations of mixed Hodge structures over \(S^{an}\) of type \((0,0)\), \((-1,0)\), \((0,-1)\), \((-1,-1)\). In general, the corresponding statement holds only up to isogeny.

	Finally, we introduce the global Mumford--Tate group of a 1-motive over $S$ and show that its neutral connected component identifies with the Mumford--Tate group of the generic fiber.

\end{abstract}


\maketitle


\tableofcontents

\section*{Introduction}

Let $S$ be a connected scheme smooth and of finite type over $\CC,$ and denote by $S^\an$ the associated complex analytic space. 

 In \cite[Construction (10.1.3)]{D1} Deligne proved that there is an equivalence between 
 \begin{enumerate}
 	\item the category of 1-motives over $\CC$, and 
 	\item the category of torsion-free and polarizable variations of mixed Hodge structures of type (0,0), (-1,0), (0,-1), (-1,-1).
 \end{enumerate}

In the relative setting over $S$, the natural Hodge-theoretic objects are admissible variations of mixed Hodge structures. In particular, to every 1-motive $M$ over $S$, André associated in \cite[Lemma 5]{A25} a torsion-free, graded-polarizable and \emph{admissible} variation of mixed Hodge structures $\T_{\h}(M)$ of type (0,0), (-1,0), (0,-1), (-1,-1) over $S^\an,$ called its enriched Hodge realization. Note that the admissibility condition appears only in the mixed and relative setting: it is absent both in the case over $\CC$ (see Deligne's equivalence) and in the pure case (indeed, by \cite[Rappel (4.4.3)]{D0}, the category of abelian schemes over $S$ is equivalent to the category of torsion-free and graded-polarizable variations of mixed Hodge structures over $S^\an$ of type (-1,0), (0,-1)).

The main goal of this paper is to establish an analogue of Deligne's equivalence in the relative setting, namely for 1-motives over $S$.

In this paper we assume that extensions of abelian $S$-schemes by $S$-tori have no fixed part, even after pull-back to a finite étale
covering of $S.$

Using the Graber--Starr theorem on sections of an abelian scheme $A$ to reduce to the one-dimensional case treated by Zucker, Andr\'e proved in \cite[Theorem 3.2.1]{A25} the following Hodge-theoretic interpretation of sections of abelian schemes: any torsion-free, graded-polarizable and admissible variation of mixed Hodge structures over $S^{\an}$ of type $(0,0)$, $(-1,0)$, $(0,-1)$ arises, up to isogeny, from a 1-motive $[v : \mathbb{Z} \to A] $
over $S$ without toric part. By Cartier duality, Andr\'e's result implies that any torsion-free, graded-polarizable and admissible variation of mixed Hodge structures over $S^{\an}$ of type $(-1,0)$, $(0,-1)$, $(-1,-1)$ arises, up to isogeny, from a 1-motive $
[0 \to G] $
over $S$ without lattice (in fact, the Cartier dual of the semi-abelian variety $G$ is $
[\hat v : \mathbb{Z} \to \hat A],$
where $\hat A$ is the Cartier dual of the abelian scheme $A$). In loc.\ cit., Question 3.2.3, he asked whether any torsion-free, graded-polarizable and admissible variation of mixed Hodge structures over $S^{\an}$ of type $(0,0)$, $(-1,0)$, $(0,-1)$, $(-1,-1)$ arises, up to isogeny, from a 1-motive over $S$.

In Proposition \ref{prop:AndreG_m}, we show that any torsion-free, graded-polarizable and admissible variation of mixed Hodge structures over $S$ of type $(0,0)$, $(-1,-1)$ arises, up to isogeny, from a 1-motive
$[\mathbb{Z} \to \mathbb{G}_m]$
over $S$ without abelian part. Our proof, dealing with the toric case, differs substantially from Andr\'e's argument in the abelian case.

Let $\mathcal P$ be the Poincar\'e biextension of $(A,\hat A)$ by $\mathbb{G}_m$. Having a 1-motive
$M=[u:\mathbb{Z}\to G]$
is equivalent to having the two morphisms of group schemes
$v:\mathbb{Z}\to A $ and $\hat v:\mathbb{Z}\to \hat A$
and a trivialization $
\psi:\mathbb{Z}\times\mathbb{Z}\to (v\times \hat v)^*\mathcal P$
of the pull-back of the Poincar\'e biextension via $v\times \hat v$. Hence, combining Andr\'e's result \cite[Theorem 3.2.1]{A25} with our Proposition \ref{prop:AndreG_m}, we prove in Theorem \ref{Teo:AndreG} a Hodge-theoretic interpretation of sections of semi-abelian varieties,
thereby providing a positive answer to Andr\'e's question \cite[Question 3.2.3]{A25}.

Theorem \ref{Teo:AndreG} leads to the following relative analogue of Deligne’s equivalence over $\mathbb{C}$, established in Theorem \ref{VMHS1mot}: namely, if either

(a) $S$ is proper or contained in a proper variety with boundary of codimension $>1$,  or

(b) $S$ has a compactification with boundary with normal crossings and the local monodromies of $(\V_\QQ, \W_*,\F^*)$ at the boundary are unipotent,

\par\noindent the enriched Hodge realization functor $M \longmapsto \mathrm{T}_{\mathrm H}(M)$
is an equivalence of categories between
	
	\begin{enumerate}
		\item the category of 1-motives over $S$, whose lattices are globally constant for the étale topology, and whose tori are split,
		\item the category of torsion-free, graded-polarizable and admissible variations of mixed Hodge structures $(\V_\ZZ,\W_*,\F^*)$ over $S^{\an}$ of type	(0,0),\,(-1,0),\,(0,-1),\,(-1,-1), such that $ \operatorname{Gr}^W_0(\V_{\mathbb Z})$ is a globally constant lattice and $ \operatorname{Gr}^W_{-2}(\V_{\mathbb Z}) $ is a globally constant variation of Hodge structures isomorphic to a direct sum of copies of $ \mathbb Z(1) .$
	\end{enumerate}

	\par\noindent In general, this equivalence holds only up to isogeny. 
	
Finally we introduce the notion of global Mumford-Tate group $\mathrm{G}_{\vmhs^\ad (S)}(M)$ of a 1-motive $M$ over $S$ and using the theorem of the fixed part, we prove that the neutral connected component of $\mathrm{G}_{\vmhs^\ad (S)}(M)$ identifies with the Mumford-Tate group of the fiber $M_{s_0}$ at the generic point $s_0 $ of $ S$ (see Proposition \ref{MT}).

\section*{Acknowledgement}

The author thanks Yves André for bringing \cite[Theorem 3.2.1 and Question 3.2.3]{A25} to her attention and for his comments which greatly improved this paper.

\section*{Notation}

Let $S$ be an arbitrary scheme. A smooth 1-motive (simply called 1-motive) $M=(X,A,T,G,u)$ over $S$ consists of
\begin{itemize}
	\item  an $S$-group scheme $X$ which is locally for the \'etale
	topology a constant group scheme defined by a finitely generated free
	$\ZZ \,$-module,
	\item an extension $G$ of an abelian $S$-scheme $A$ by an $S$-torus $T$ with cocharacter group $X_*(T)={\uHom}({\GG}_m,T)$ and character group $X^*(T)={\uHom}(T,{\GG}_m)$,
	\item a morphism $u:X \rightarrow G$ of $S$-group schemes.
\end{itemize}

The weight filtration $W_*$ on $M $ is ${\W}_{i}(M) = M$ for $i \geq 0$, ${\W}_{-1}(M)=G, {\W}_{-2}(M)= T$ and ${\W}_{j}(M) = 0$ for $j \leq -3$. The quotients ${\rm Gr}_{i}^{{\W}}= {\W}_i / {\W}_{i-1}$ for $i=0,-1,-2$ are $X,A$ and $T$ respectively.

A 1-motive $M=(X,A,T,G,u)$ can be viewed also as a length 1 complex $[u: X \to G]$ of abelian sheaves for the \emph{fppf} topology concentrated in degrees -1 and 0. Let $M_1=[u_1:X_1 \to G_1]$ and $M_2=[u_2 : X_2 \to G_2]$ be two 1-motives defined over $S.$
	A morphism $(f_{X}:X_1 \to X_2,f_G:G_1 \to G_2)$ from $M_1$ to $M_2$ is a morphism of complexes of abelian sheaves.
Since any arrow from $G$ to $X$ is zero, there are no homotopies between 1-motives to consider. 
In particular, \textit{1-motives form a category and not a 2-category}.
Denote by 
\[1-\mot (S)\]
 the category of 1-motives over $S$. It is an additive category but it is not an abelian category.

 An isogeny between two 1-motives
 $M_{1}=[u_1: X_{1} \to G_{1}]$ and
 $M_{2}=[u_2: X_{2} \to G_{2}]$ is a morphism $(f_{X},f_{G})$ such that
 $f_{X}:X_{1} \rightarrow X_{2}$ is injective with finite cokernel, and
 $f_{G}:G_{1} \rightarrow G_{2}$ is surjective with finite kernel. 

\section{Hodge-theoretic interpretation of sections of semi-abelian varieties}

	Let $S$ be a \emph{connected scheme smooth and of finite type over $\CC$.} Denote by $S^{\an}$ the associated complex analytic space. 

A variation of mixed Hodge structures $(\V_\ZZ,\W_*,\F^*)$ over $S^{\an}$ consists of a local system $\V_\ZZ$ of finitely generated $\ZZ$-modules over $S^{\an}$, an increasing weight filtration $\W_*$ of the local system $\V_\QQ = \V_\ZZ \otimes_\ZZ \QQ$ of $\QQ$-vector spaces by local sub-systems of $\QQ$-vector spaces, and a decreasing Hodge filtration $\F^*$ of the holomorphic vector bundle
$\V_{\cO_{S^\an}}=\V_{\ZZ} \otimes_\ZZ \cO_{S^\an}$ by holomorphic vector subbundles, such that 
\begin{enumerate}
	\item the  filtration $\F^*$ satisfies the transversality axiom, i.e. if $\nabla$ denotes the integrable canonical connection of the holomorphic vector bundle $\V_{\cO_{S^\an}} $ (see \cite[I Theorem 2.23]{D70}), we have that
	$\nabla \F^i ( \V_{\cO_{S^\an}}) \subset \Omega^1_{S^\an} \otimes_{\cO_{S^\an}} \F^{i-1}( \V_{\cO_{S^\an}})$, and
	\item over each point $s$ of $S^\an$, the fibre $(\V_\ZZ,\W_*,\F^*)_s$ is a mixed Hodge structure. 
\end{enumerate}
A morphism of variations of mixed Hodge structures is a morphism of local systems which respects the increasing weight filtration $\W_*$ and the  decreasing Hodge filtration $\F^*$ pointwise.
Denote by 
\[\vmhs (S)\]
 the category of variations of mixed Hodge structures 
over $S^\an$. It is a Tannakian category over $\QQ$.

A variation of mixed Hodge structures $(\V_\ZZ,\W_*,\F^*)$ over $S^{\an}$ is called
\begin{itemize}
	\item torsion-free if over each point $s$ of $S^\an,$ the fibre $(\V_\ZZ,\W_*,\F^*)_s$ is a torsion-free mixed Hodge structure (i.e. the underlying $\ZZ$-module $(\V_\ZZ)_s$ is torsion-free).
	\item graded-polarizable if for each integer $n,$ there exists a morphism of local systems of $\QQ$-vector spaces 
	$\beta_n: \Gr_n^\W (\V_\QQ) \otimes \Gr_n^\W (\V_\QQ) \rightarrow \QQ(-n) ,$
	which defines over each point $s$ of $S^\an$ a polarization of the fibre $(\Gr_n^\W (\V_\QQ),\W_*,\F^*)_s$ (i.e. over each point $s$ of $S^\an$, the real bilinear form 
	$(2 \pi i)^n \beta_s (-,I-)$ on $\Gr_n^\W (\V_\QQ) \otimes \RR$ is symmetric and positive-definite, with $I$ the image of $i$ via the isomorphism $ \mathrm{Res}_{\CC / \RR}(\GG_m)( \RR) \rightarrow \CC^*$).
	\item admissible if its restriction to any curve $C \subset S$ satisfies the usual admissibility conditions at infinity: after reducing to the unipotent local monodromy case,
	there exists a compactification $ C^\an \subset \overline{C^\an}$ such that $\Sigma :=\overline{C^\an} - C^\an $ is a divisor with normal crossings, and the conditions (1)–(2) of the original definition (see \cite[(3.13)]{SZ} or \cite{K} and \cite[ \S 3.2]{Gao}) are satisfied.
	
	
	
\end{itemize}

\par\noindent Denote by 
\[\vmhs^\ad (S)\]
the category of admissible variations of mixed Hodge structures 
over $S^\an$. It is a full Tannakian subcategory of the Tannakian category $\vmhs (S)$ (see \cite[Appendix A]{SZ}).

\medspace

We start by proving the \emph{Hodge-theoretic interpretation of (holomorphic versus algebraic) sections of tori}. In \cite[Theorem 3.2.1]{A25} André treated the abelian analogue. The toric and abelian proofs differ significantly.

\medspace

We first treat the case where the $S$-torus $T=\GG_m$
is split and the lattice $X=\ZZ$ is globally constant for the étale topology. The general case is then obtained in Proposition~\ref{prop:AndreG_m}(3) by passing to a finite étale covering of $S.$
Observe that $\GG_m(S) = \Hom_s(S,\GG_m) =\Gamma (S, \cO_S^*)$ and analytically $\GG_m^\an(S^\an)  =\Gamma (S^\an, \cO_{S^\an}^*).$ Denote by $\ZZ(1)= 2 \pi \ii \ZZ$ the constant variation of pure Hodge structure over $S^\an$ of type (-1,-1) associated to $T=\GG_m,$ and by $\ZZ_S$  the constant variation of pure Hodge structure over $S^\an$ of type (0,0) associated to $X=\ZZ.$ Over $S^\an$ there is the exponential exact sequence 
$0 \to \ZZ(1) \to \cO_{S^\an} \stackrel{\exp}{\to} \cO_{S^\an}^* \to 1 . $ Taking global sections gives the long exact sequence
\begin{equation}\label{eq:GlobalSections}
	0 \longrightarrow 2 \pi \ii \ZZ  \longrightarrow \Gamma (S^\an, \cO_{S^\an}) \stackrel{\exp}{\longrightarrow} \Gamma (S^\an, \cO_{S^\an}^*) \stackrel{\delta}{\longrightarrow} \HH^1(S^\an, \ZZ(1)) \longrightarrow \HH^1(S^\an, \cO_{S^\an}) 
\end{equation}
 Elements of $\HH^1(S^\an, \ZZ(1))$ parametrize extensions of $\ZZ$ by $\ZZ(1)$ as local systems:
\[\HH^1(S^\an, \ZZ(1)) \cong \Ext^1_{\locsys}(\ZZ_S,\ZZ(1) ) ,\]
while elements of $ \Gamma (S^\an, \cO_{S^\an}^*)$ parametrize extensions of $\ZZ$ by $\ZZ(1)$ as variations of mixed Hodge structures:
\[\GG_m^\an(S^\an) \cong  \Ext^1_{\vmhs(S)}(\ZZ_S,\ZZ(1) ) .\]
The connecting map $\delta$ is the map \emph{forgetting the Hodge structure} and we have the short exact sequence
\[ 0 \to \exp (\Gamma (S^\an, \cO_{S^\an}) ) \to \Ext^1_{\vmhs(S)}(\ZZ_S,\ZZ(1) ) \stackrel{\delta}{\longrightarrow} \Ext^1_{\locsys}(\ZZ_S,\ZZ(1))  .\]
By the normality Theorem of \cite[\S 5]{A}, the restriction $\delta_|$ of the connecting map $\delta$ to admissible variations of mixed Hodge structures gives rise to an injection
\begin{equation}\label{eq:Delta}
	\Ext^1_{{\vmhs}^\ad(S)}(\ZZ_S,\ZZ(1) ) \stackrel{\delta_|}{\longhookrightarrow} \HH^1(S^\an, \ZZ(1))
\end{equation}
A global algebraic unit section $\sigma \in \GG_m(S)  =\Gamma (S, \cO_S^*)$ furnishes an admissible variation of mixed Hodge structures in 
$\Ext^1_{{\vmhs}^\ad(S)}(\ZZ_S,\ZZ(1) ):$ in fact, to $\sigma$ is associated the 1-motive $M=[u:\ZZ \to \GG_m ], u(1) =\sigma $ and the Hodge realization of the short exact sequence $0 \to \W_{-1}(M) \to M \to \Gr_{0}(M) \to 0$ lies in
$\Ext^1_{{\vmhs}^\ad(S)}(\ZZ_S,\ZZ(1) )$ because of \cite[Lemma 5]{A}. Fiberwise the admissible variation of mixed Hodge structure in $\Ext^1_{{\vmhs}^\ad(S)}(\ZZ_S,\ZZ(1) )$  associated to the 1-motive $M_s=[u_s:\ZZ \to \GG_{m, s} ], u_s(1) =\sigma(s) $ splits only if $\sigma(s)$ is trivial. Hence one gets injections
\begin{equation} \label{eq:exsecTorus}
	\GG_m(S) \longhookrightarrow  \Ext^1_{{\vmhs}^\ad(S)}(\ZZ_S,\ZZ(1) ) \longhookrightarrow \HH^1(S^\an, \ZZ(1)) .
\end{equation}
In the statements $(1)$ and $(2)$ of the following Proposition, we still suppose $X=\ZZ$ and $T=\GG_m$.

\begin{proposition}\label{prop:AndreG_m} Let $S$ be a connected scheme smooth and of finite type over $\CC.$
	\par\noindent (1)	If $S$ is also proper over $\CC,$ then
\[\Ext^1_{{\vmhs}(S)}(\ZZ_S,\ZZ(1) )=\Ext^1_{{\vmhs}^\ad(S)}(\ZZ_S,\ZZ(1) )=0.\]
\par\noindent (2) In general, the inclusions \eqref{eq:exsecTorus} are isomorphisms \footnote{The first isomorphism is the enriched Hodge realization of the isomorphism \cite[(2.4.1)]{D89} }
	\[	\GG_m(S) \cong \Ext^1_{{\vmhs}^\ad(S)}(\ZZ_S,\ZZ(1) ) \cong \HH^1(S^\an, \ZZ(1))^{(0,0)}. \]
	\par\noindent (3) If $T$ is a 1-dimensional $S$-torus with cocharacter group $X_*(T)$ and $X$ is an $S$-group scheme which is locally for the étale topology a constant group scheme defined by a rank 1 free $\ZZ$-module, then
	\[T(S) \otimes \QQ \cong  \Ext^1_{{\vmhs}^\ad(S)}(X \otimes \QQ,X_*(T) \otimes \QQ(1) ) \cong \HH^1(S^\an, X_*(T) \otimes \QQ(1))^{(0,0)}  .\]
\end{proposition}
\begin{proof}
	(1) The properness of $S$ over $\CC$ and \cite[Exposé XII Corollaire 4.3]{SGA1} imply that algebraic and holomorphic global unit sections of $S$ coincide, that is $\GG_m(S)= \GG_m^\an(S^\an).$ Hence we get the first equality.
	Since $S$ is connected and proper over $\CC$ (analytically $S^\an$ is compact and connected), by
	Grothendieck's Finiteness Theorem (analytically by Liouville Theorem) its global sections are only the constants, i.e. $\Gamma(S,\cO_S)= \CC$ (analytically  $\Gamma(S^\an,\cO_{S^\an})= \CC$). In particular its global unit sections are the constants: $\GG_m(S)=\Gamma(S,\cO_S^*)= \CC^*$ (analytically $\GG_m^\an(S^\an)=\Gamma(S^\an,\cO_{S^\an}^*)= \CC^*$). Any constant $c \in \Gamma(S^\an,\cO_{S^\an}^*)= \CC^*$ has a global holomorphic logarithm, and so its class $\delta(c)$ in $\HH^1(S^\an, \ZZ(1))$ is zero. In other words the connecting map $\delta$ in the long exact sequence \eqref{eq:GlobalSections} is the zero map. Since its restriction $\delta_|$ \eqref{eq:Delta} is injective, we get the second equality. 

(2) By Hironaka, there exists a smooth compactification $\overline{S^\an}$ of $S^\an$ with boundary with normal crossings. Denote $j: S^\an \hookrightarrow \overline{S^\an}$ the open immersion and $\Sigma :=\overline{S^\an} - S^\an =\cup_{i \in I}D_i.$ Recall that we have the exact sequence
\[ 0 \longrightarrow \HH^1(\overline{S^\an}, \ZZ(1)) \stackrel{j^*}{\longrightarrow} \HH^1(S^\an, \ZZ(1)) \stackrel{\mathrm{Res}}{\longrightarrow} 
\oplus_{i \in I} \ZZ \stackrel{\mathrm{Cl}}{\longrightarrow} \HH^2(\overline{S^\an}, \ZZ(1)). \]
The piece $\Gr^W_{1} \HH^1(S^\an, \ZZ(1))$ of pure weight 1 consists of in $\HH^1(S^\an, \ZZ(1))$ with trivial residue and the piece  $\Gr^W_{0} \HH^1(S^\an, \ZZ(1))$ of pure weight 0, which is only of type (0,0), consists of residues lying in the kernel of the cycle class map $\mathrm{Cl}.$ We can see a class $(n_i)_i$ in $\HH^1(S^\an, \ZZ(1))^{(0,0)}$ as a weight 0 divisor with integral coefficients and support in $\Sigma$.

\par\noindent Since the class $\delta_| (\sigma)$ in $\HH^1(S^\an, \ZZ(1))$ of any global algebraic unit section $\sigma \in \GG_m(S)$ is of type (0,0), the target of the last inclusion \eqref{eq:exsecTorus} is $\HH^1(S^\an, \ZZ(1))^{(0,0)} $ ($\sigma \in \GG_m(S)$ is a function without zero and poles in $S$, that is it defines a divisor with integral coefficients whose support is in $\Sigma$). Note that if $\mathrm{Res}\delta_| (\sigma)=0$, $\sigma $ is a constant and $\delta_| (\sigma)$ is the class of the trivial extension of $\ZZ_S$ by $\ZZ(1)$ (the class $\delta_| (\sigma)$ lies in $\im (j^*)$ and we are in case (1)).
\par\noindent  The classes $(n_i)_i$ in $\HH^1(S^\an, \ZZ(1))^{(0,0)}$ form a free subgroup of 
$\oplus_{i \in I} \ZZ$ of finite rank because of the condition $\mathrm{Cl}(n_i)_i=0.$ Consider the exact sequence
\[ 0 \longrightarrow \GG_m(S) \stackrel{\mathrm{Div}}{\longrightarrow} \mathrm{Div}^0_\Sigma(\overline{S}) \stackrel{\phi}{\longrightarrow} \mathrm{Pic}^0(\overline{S}),\]
where $\mathrm{Div}^0_\Sigma(\overline{S})$ are weight 0 divisors with support in $\Sigma$. The image of $\phi$ is a discrete subgroup of the compact abelian variety $\mathrm{Pic}^0(\overline{S})$ and so it is finite. But any finite subgroup of  a free group of finite rank is trivial, that is $\phi=0$. In other words, any weight 0 divisors with support in $\Sigma$, that is any class  in $\HH^1(S^\an, \ZZ(1))^{(0,0)} ,$ comes from a global algebraic unit section $\sigma \in \GG_m(S)$\footnote{Consider a class $(n_i)_i$ in $\HH^1(S^\an, \ZZ(1))^{(0,0)}.$ The equality $\mathrm{Cl}(n_i)_i=0$ is a gluing condition and it is necessary for the existence of a global unit section $\sigma \in \GG_m(S):$ $\sigma$ extends to a well-defined meromorphic function on $\overline{S}$ with divisor supported in $\Sigma$. We have proved the converse: any meromorphic function on $\overline{S}$ with divisor supported in $\Sigma$ comes from a global unit section on $S$.	As pointed out by André in \cite[3.2.2 Remarks 1)]{A25} the condition of admissibility has to do with meromorphy at infinity rather than holomorphy.}. 

(3) Consider a finite étale Galois covering $p:S' \to S$ such that $T \times_S S'=\GG_m$ and $X \times_S S'=\ZZ.$ Denote by $\Gamma$ the Galois group $ \mathrm{Gal}(S'/S).$ For global sections we have the equality $T(S)=\GG_m(S')^\Gamma,$ and for cohomology classes we have the isomorphism $ \HH^1(S^\an, X_*(T) \otimes \QQ(1))^{(0,0)} \cong (\HH^1(S'^{\; \an}, \QQ(1))^{(0,0)})^\Gamma.$ Step (2) furnishes the $\Gamma$-equivariant isomorphism $\GG_m(S') \cong \HH^1(S'^{\; \an}, \ZZ(1))^{(0,0)}.$ Taking the $G$-invariants and tensorizing with $\QQ$ we get the isomorphisms
\[T(S) \otimes \QQ = \GG_m(S')^\Gamma \otimes \QQ \cong \big(\HH^1(S'^{\; \an}, \QQ(1))^{(0,0)}\big)^\Gamma \cong \HH^1(S^\an, X_*(T) \otimes \QQ(1))^{(0,0)}.\]
\end{proof}

Combining the abelian case, established by André in \cite[Theorem 3.2.1]{A25}, with the toric case, treated in Proposition \ref{prop:AndreG_m}, we now turn to the \emph{Hodge-theoretic interpretation of (holomorphic versus algebraic) sections of semi-abelian varieties}. 

\medspace

 Let $G$ be an extension of an abelian $S$-scheme by a 1-dimensional $S$-torus: $0 \to T \to  G \to A \to 0.$
We first treat the case where the torus underlying $G$
is split and the lattice is globally constant for the étale topology, namely $T=\GG_m$ and $X=\ZZ.$ The general case is then obtained in Theorem~\ref{Teo:AndreG} (2) by passing to a finite étale covering of $S.$  Let $(\V_\ZZ, \W_*,\F^*)$ be the torsion-free, graded-polarizable and admissible variation of mixed Hodge structure over $S^\an$ of type (-1,0), (0,-1) and (-1,-1) associated to $G$ (see \cite[Lemma  5]{A}). We have that $\W_{-2}(\V_\ZZ)$ is the constant variation of pure Hodge structure $\ZZ(1)= 2 \pi \ii \ZZ$ of type (-1,-1) associated to the torus $\GG_m,$ and $\Gr^\W_{-1}(\V_\ZZ)$ is the torsion-free, graded-polarizable variation of mixed Hodge structure of type (-1,0) and (0,-1) associated to the abelian scheme $A.$
Denote by $\ZZ_S$  the constant variation of pure Hodge structure of type (0,0) associated to the lattice $\ZZ.$ 
Note that $G(S)=\Hom_S(S,G)$ and $G^\an(S^\an)=\Hom_{S^\an}(S^\an,G^\an).$

Elements of $\HH^1(S^\an, \V_\ZZ)$ parametrize extensions of $\ZZ$ by $\V_\ZZ$ as local systems:
\[\HH^1(S^\an, \V_\ZZ) \cong \Ext^1_{\locsys}(\ZZ_S,\V_\ZZ ) ,\]
while elements of  $G^\an(S^\an)$ parametrize extensions of $\ZZ$ by $\V_\ZZ$ as variations of mixed Hodge structures:
\begin{equation} \label{eq:AnaliticSection}
	G^\an(S^\an) \cong \Ext^1_{\vmhs (S)}(\ZZ_S,\V_\ZZ ) .
\end{equation}
Consider the map \emph{forget the Hodge structure}
\[\Ext^1_{\vmhs(S)}(\ZZ_S,\V_\ZZ ) \to \Ext^1_{\locsys}(\ZZ_S,\V_\ZZ ) \cong  \HH^1(S^\an, \V_\ZZ) .\]
By the normality Theorem of \cite[\S 5]{A}, the restriction of this map to admissible variations of mixed Hodge structures gives rise to an injection
\[\Ext^1_{{\vmhs}^\ad(S)}(\ZZ_S,\V_\ZZ ) \longhookrightarrow \HH^1(S^\an, \V_\ZZ)^{(0,0)} .\]
A global algebraic section $g \in G(S)$ furnishes an admissible variation of mixed Hodge structures in 
$\Ext^1_{{\vmhs}^\ad(S)}(\ZZ_S,\V_\ZZ):$ in fact, to $g$ is associated the 1-motive $M=[u:\ZZ \to G ], u(1) =g,$  and the Hodge realization of the short exact sequence $0 \to \W_{-1}(M) \to M \to \Gr_{0}(M) \to 0$ lies in
$\Ext^1_{{\vmhs}^\ad(S)}(\ZZ_S,\V_\ZZ )$ because of \cite[Lemma 5]{A}. Since the class in $\HH^1(S^\an, \V_\ZZ)$ of any algebraic section is of type $(0,0)$ one gets the injections
\begin{equation}\label{eq:exsecG}
	 G(S) \longhookrightarrow  \Ext^1_{{\vmhs}^\ad(S)}(\ZZ_S,\V_\ZZ) \longhookrightarrow \HH^1(S^\an, \V_\ZZ)^{(0,0)} 
\end{equation}
(the injectivity on the left map can be checked fiberwise and it reduces to the fact that the admissible variation of mixed Hodge structure in $\Ext^1_{{\vmhs}^\ad(S)}(\ZZ_S,\V_\ZZ)$  associated to the 1-motive $M_s=[u_s:\ZZ \to G_s ], u_s(1) =g(s), $ splits only if $g(s)$ is trivial).

\medspace

An extension $G/S$ of an abelian $S$-scheme by a $S$-torus is said to have \emph{no fixed part, even after pull-back to a finite étale
covering of $S$}, if for every finite étale covering $	p:S'\to S$
one has	$	H^0(S'^{\mathrm{an}},p^\ast \V_{\mathbb Z})=0,	$
or equivalently	$(p^\ast \V_{\mathbb Z , s})^{\pi_1(S'^\an , s')}=0,	$ where  \(\V_{\mathbb Z}\) denotes the local system associated
with \(G\).
Geometrically, this is equivalent to requiring that \(G/S\) has no non-trivial
constant semi-abelian subvariety, even after pull-back to any finite étale
covering of \(S\).

\begin{theorem}\label{Teo:AndreG}
	Let $S$ be a connected scheme smooth and of finite type over $\CC.$ Let $G$ be an extension of an abelian $S$-scheme by a 1-dimensional $S$-torus $T$. 	Assume that $G/S$ has no fixed part, even after pull-back to a finite étale
	covering of $S.$
\par\noindent (1) Suppose $X = \ZZ$ and $T= \GG_m.$ If either

- $S$ is proper or contained in a proper variety with boundary of codimension $>1$,  or

- $S$ has a compactification with boundary with normal crossings and the local monodromies of $(\V_\QQ, \W_*,\F^*)$ at the boundary are unipotent,

\par\noindent the inclusions \eqref{eq:exsecG} are isomorphisms \footnote{The first isomorphism is the enriched Hodge realization of the isomorphism \cite[(2.4)]{D89}}
\[	G(S) \cong \Ext^1_{{\vmhs}^\ad(S)}(\ZZ_S,\V_\ZZ ) \cong \HH^1(S^\an, \V_\ZZ)^{(0,0)}. \]

\par\noindent (2) Let $X$ be an $S$-group scheme which is locally for the étale topology a constant group scheme defined by a rank 1 free $\ZZ$-module. Then
\[G(S) \otimes \QQ \cong  \Ext^1_{{\vmhs}^\ad(S)}(X \otimes \QQ, \V_\QQ ) \cong \HH^1(S^\an, \V_\QQ)^{(0,0)}  .\]
\end{theorem}

\begin{proof}
	(1) According to \cite[Theorem 3.2.1]{A25} and Proposition \ref{prop:AndreG_m} (2), we have the following commutative diagram with exact rows
			\[
			\begin{array}{ccccc}
				\GG_m(S)  & \rightarrow &G(S)  & \stackrel{\pi}{\rightarrow} & A(S)  \\[1mm]
				\cong &&\hookdownarrow&&\cong  \\[2mm]
				\Ext^1_{{\vmhs}^\ad(S)}(\ZZ_S,\ZZ(1) ) &\rightarrow& \Ext^1_{{\vmhs}^\ad(S)}(\ZZ_S,\V_\ZZ )&\stackrel{\pi_*}{\rightarrow}&  \Ext^1_{{\vmhs}^\ad(S)}(\ZZ_S,\Gr^\W_{-1}(\V_\ZZ) )  \\[1mm]
				\cong &&\hookdownarrow&&\cong \\[2mm]
				\HH^1(S^\an, \ZZ(1))^{(0,0)} &  \rightarrow & \HH^1(S^\an, \V_\ZZ)^{(0,0)} &\stackrel{\pi_*}{\rightarrow}  &  \HH^1(S^\an, \Gr^\W_{-1}(\V_\ZZ))^{(0,0)}\\
			\end{array}
			\]	
\par\noindent
Using the isomorphisms in the columns of this diagram, we identify the corresponding elements.

Let $x$ be a class in $\HH^1(S^\an,\V_\ZZ)^{(0,0)}$. If $\pi_*(x)=0,$ there exists a global algebraic section $t \in \GG_m(S) \subseteq G(S)$ whose class in $\HH^1(S^\an,\V_\ZZ)^{(0,0)}$ is $x.$ 
Suppose now $\pi_*(x) \not=0.$
Denote by $E$ an extension representing $x$. 
This extension fits into the following commutative diagram:
\[
\begin{array}{ccccccccc}
		& & 0 && 0 && 0 \\
	&  &\downarrow &  & \downarrow&& \downarrow \\[2mm]
	0 & \to & \W_{-2}(\V_\ZZ) & \to & \V_\ZZ & \to & \Gr_{-1}^W (\V_\ZZ) & \to & 0 \\
	& & \downarrow && \downarrow && \downarrow \\
	0 & \to & E'' & \to & E & \to & E' & \to & 0 \\
	& & \downarrow && \downarrow && \downarrow \\
	 &  & \mathbb{Z}_S & = & \mathbb{Z}_S & = & \mathbb{Z}_S &  &  \\
	& & \downarrow && \downarrow && \downarrow \\
	& & 0 && 0 && 0
\end{array}
\]
The extensions $E''$ and $E'$ belong to the category $\vmhs(S)$, which is closed under extensions. Hence $E$ also belongs to $\vmhs(S)$, and therefore $ x \in \Ext^1_{\vmhs(S)}(\ZZ_S,\V_\ZZ).$ By \eqref{eq:AnaliticSection} there exists then a global analytic section $g^\an \in	G^\an(S^\an)$ corresponding to the class $x.$

\par\noindent  Set $\pi_*(x) =a$ with $a \in A(S).$ To the global algebraic section $a$ is associated the 1-motive $[v:\ZZ \to A], v(1)=a,$ and the pull-back $a^*G$ of the extension $G$ via $a:S \to A.$ The $\GG_m$-torsor $a^*G$ is trivial if and only if the global algebraic section $a \in A(S)$ lifts to a global algebraic section $ \in G(S)$ living above $a$.
The Cartier dual of the extension $G$ is the 1-motive  $[\hat{v}:\ZZ \to \widehat{A}], \hat{v}(1)=\hat{a}, $ where $\widehat{A}$ is the Cartier dual of the abelian scheme $A$ and $\hat{a} $ is the global algebraic section of $\widehat{A}$ which parametrizes the extension $G.$   
 \par\noindent Let $\mathcal P$ the Poincaré biextension of $(A,\widehat{A})$ by $\GG_m$.
	The 
	$\GG_m$-torsor $a^*G$ is the fibre ${\mathcal P}_{a,\hat{a}}$ of $\mathcal P$ above the point $(a,\hat{a}) \in A \times \widehat{A}.$
	Having the global analytic section $g^\an \in	G^\an(S^\an)$ is equivalent to have an analytic trivialization $\psi^\an: \ZZ \times \ZZ \to (v \times \hat{v})^* {\mathcal P}$ of the pull-back $(v \times \hat{v})^* {\mathcal P}$ of the Poincaré biextension via $v \times \hat{v}.$ Being the pull-back $(v \times \hat{v})^* {\mathcal P} $ the trivial $\GG_m$-torsor $\ZZ \times  \ZZ \times \GG_m,$ we can see the trivialization $\psi^\an$ as a map from $\ZZ \times \ZZ$ to $\GG_m.$ But then, by Proposition \ref{prop:AndreG_m} (2), the trivialization $\psi^\an$ has an algebraic origin given by 
	an algebraic trivialization $\psi$ of the pull-back $(v \times v^*)^* {\mathcal P}:$ in other words, there exists a global algebraic section $t \in \GG_m(S)$ such that fiberwise, for any $s \in S,$ $\psi_s(1,1)$ defines a point $(1,1,t(s))$ in  $((v \times \hat{v})^* {\mathcal P}_s)_{1,1}$ which in turn furnishes a point $g(s)$ in ${\mathcal P}_{a(s),\hat{a}(s)}= a(s)^*G_s.$ Hence the $\GG_m$-torsor $a^*G$  is trivial, that is
	there exists a global algebraic section $g \in G(S) $ living above $a$, whose class in $\HH^1(S^\an, \W_{-1}(\V_\ZZ))^{(0,0)}$ is $x.$

(2) Consider a finite étale Galois covering $p:S' \rightarrow S $
such that $T \times_S S'=\GG_m, X \times_S S'=\ZZ,$ and the local
monodromies of $p^\ast \V_{\mathbb Q}$ at the boundary are unipotent. Let
$\Gamma =\operatorname{Gal}(S'/S),$ let \(s'\in S'^{\an}\) be a point above
\(s\in S^{\an}\), and set
$ \Pi :=\pi_1(S^{\an},s),
\Pi' :=\pi_1(S'^{\an},s'). $
We identify \(\Pi'\) with the kernel of the natural map \(\Pi\to \Gamma\). Set $G_{S'}=G\times_S S'.$
 For global sections we have the equality $G(S)= G_{S'}(S')^\Gamma.$ For cohomology classes the short exact sequence 
\[ 0 \longrightarrow \HH^1 \big(\Gamma, (p^\ast \V_{\mathbb Z , s})^{\pi_1(S'^\an , s')} \big) \longrightarrow \HH^1(\Pi,\V_{\ZZ ,  s}) \longrightarrow  \HH^1(\Pi', p^\ast \V_{\mathbb Z , s})^\Gamma \]
yields the inclusion 
$
\HH^1(\Pi,\V_{\ZZ, s}) \hookrightarrow \HH^1(\Pi', p^*\V_{\ZZ, s})^\Gamma,
$
since by hypothesis $G$ has no fixed part, even after pull-back to a finite étale cover of $S$ (recall that $\HH^1(\Pi,\V_{\ZZ, s})= \HH^1(S^\an, \V_\ZZ)$ and $\HH^1(\Pi', p^*\V_{\ZZ, s})=\HH^1(S'^\an, p^*\V_\ZZ)$ ).  Passing to the \((0,0)\)-parts, the inclusion still holds
\[
\HH^1(\Pi,\V_{\ZZ, s})^{(0,0)}
\longhookrightarrow
\bigl(\HH^1(\Pi', p^*\V_{\ZZ, s})^{(0,0)}\bigr)^\Gamma.
\]
On the covering \(S'\), the case (1) provides the isomorphism $G_{S'}(S') \cong \HH^1(\Pi', p^*\V_{\ZZ, s})^{(0,0)}.$
Taking \(\Gamma\)-invariants gives
$ G(S) \otimes_{\ZZ}\QQ= G_{S'}(S')^\Gamma \otimes_{\ZZ}\QQ \cong
\bigl( \HH^1(\Pi', p^*\V_{\QQ, s})^{0,0}\bigr)^\Gamma.$
Thus we obtain 
\[
G(S) \otimes_{\ZZ}\QQ
\longhookrightarrow
\HH^1(\Pi,\V_{\QQ, s})^{(0,0)}
\longhookrightarrow
\bigl( \HH^1(\Pi', p^*\V_{\QQ, s})^{0,0}\bigr)^\Gamma
\cong G(S) \otimes_{\ZZ}\QQ.
\]
The composition is the identity. Therefore the first arrow is an isomorphism.
\end{proof}

\section{The enriched Hodge realization functor}

	Let $S$ be a \emph{connected scheme smooth and of finite type over $\CC$.}
	
 To any 1-motive $M=[u:  X  \to  G]$ over $S,$ André associated in \cite[Lemma 5]{A25} a torsion-free, graded-polarizable and admissible variations of mixed Hodge structures $(\V_\ZZ,\W_*,\F^*)$ over $S^{\an}$ of type	(0,0),\,(-1,0),\,(0,-1),\,(-1,-1), defined as follows:
\begin{itemize}
	\item  the local system $\V_\ZZ$ of finitely generated free $\ZZ$-modules over $S^{\an}$ is the fibred product $ \li G \times_G X$ of $\li G$ and $X$ over $G,$
	where $\li G$ and $X$ are viewed over $G$ via the exponential map and the morphism $u$ respectively,
	\item the increasing weight filtration $\W_*$ of $\V_\QQ = \V_\ZZ \otimes_\ZZ \QQ$ by local sub-systems of $\QQ$-vector spaces is defined over $\ZZ$ by 
	\begin{eqnarray}
		\nonumber \W_0 (\V_\ZZ)&=&\V_\ZZ, \\
		\nonumber \W_{-1} (\V_\ZZ) &=& \ker \big(\exp: \li G \rightarrow G \big) = \R_1 f_*^{\an} {\ZZ}, \\
		\nonumber \W_{-2} (\V_\ZZ) &\cong& X_*(T) \otimes \ZZ(1),
	\end{eqnarray}
	where ${\R}_1f_*^{\an}{\ZZ}$ is the local system over $S^{\an}$ of the first homology group of the fibres of the structural morphism $f: G \rightarrow S$ of the extension $G$ underlying the 1-motive $M$. In particular, over each point $s$ of $S^\an$, the fibre $({\R}_1 f_*^{\an}{\ZZ})_s$ is isomorphic to $ \h_1(G_s,{\ZZ})$.
	\item the decreasing Hodge filtration $\F^*$ of $\V_{\cO_{S^\an}}=\V_{\ZZ} \otimes_\ZZ \cO_{S^\an}$ by holomorphic vector sub-bundles is defined by  $\F^0(\V_{\cO_{S^\an}}) = \ker \big( \alpha_{\cO_{S^\an}}: \V_{\cO_{S^\an}} \rightarrow \li G \big)$ where $ \alpha_{\cO_{S^\an}}$ is obtained from $\alpha: \V_\ZZ \rightarrow \li G$ by extending the scalars to $\cO_{S^\an}$.
\end{itemize}

 The \emph{Hodge realization} $\T_{\mathbb{Q}}(M)$ of the 1-motive $M$ is the local system $
\V_{\mathbb{Q}} = (\operatorname{Lie} G \times_G X)\otimes_{\mathbb{Z}} \mathbb{Q}$
of $\mathbb{Q}$-vector spaces over $S^{\mathrm{an}}$ underlying the variation of mixed Hodge structure
$ ({\V}_\ZZ, \W_*, \F^*),$
and \emph{its integral structure} $\T_{\mathbb{Z}}(M)$ is the local system $
\V_{\mathbb{Z}} = \operatorname{Lie} G \times_G X $
of finitely generated free $\mathbb{Z}$-modules over $S^{\mathrm{an}}.$ The \emph{enriched Hodge realization} $\T_{\h}(M)$ of $M$ is the full variation of mixed Hodge structure $ (\V_{\mathbb{Z}}, \W_*, \F^*)$
associated to $M,$ rather than merely its underlying local system $\V_{\mathbb{Q}}$. We thus obtain \emph{the enriched Hodge realization functor}
\[
\T_{\h} : 1-\mathrm{Mot}(S) \longrightarrow \mathrm{VMHS}^{\mathrm{ad}}(S) \subset \mathrm{VMHS}(S),
\qquad
M \longmapsto T_\h(M)
\]

Recall that in this paper extensions of abelian $S$-schemes by $S$-tori have no fixed part, even after pull-back to a finite étale
covering of $S.$

\begin{theorem}\label{VMHS1mot} 
	Let $S$ be a connected scheme smooth and of finite type over $\CC.$
	If either
	
	(a) $S$ is proper or contained in a proper variety with boundary of codimension $>1$,  or
	
	(b) $S$ has a compactification with boundary with normal crossings and the local monodromies of $(\V_\QQ, \W_*,\F^*)$ at the boundary are unipotent,

\par\noindent the enriched Hodge realization functor   
\[ M \longmapsto  {\T}_{\h}(M) \] 
 is an equivalence of categories between
	
	\begin{enumerate}
		\item the category of 1-motives over $S$, whose lattices are globally constant for the étale topology, and whose tori are split,
		\item the category of torsion-free, graded-polarizable and admissible variations of mixed Hodge structures $(\V_\ZZ,\W_*,\F^*)$ over $S^{\an}$ of type	(0,0),\,(-1,0),\,(0,-1),\,(-1,-1), such that $ \operatorname{Gr}^W_0(\V_{\mathbb Z})$ is a globally constant lattice and $ \operatorname{Gr}^W_{-2}(\V_{\mathbb Z}) $ is a globally constant variation of Hodge structures isomorphic to a direct sum of copies of $ \mathbb Z(1) .$
	\end{enumerate}

\par\noindent In general, namely without conditions (a) or (b) on $S$ and without the conditions on the lattices and the tori underlying 1-motives, the above equivalence holds only up to isogeny \footnote{This is due to the fact that, in the proofs of \cite[Theorem 3.2.1]{A25} and Theorem \ref{Teo:AndreG} (2) , the arguments are carried out at the integral level after passing to a finite étale covering of $S,$ and are then descended to $S$, yielding statements only up to isogeny.}.
\end{theorem}

\begin{proof}
	(1) We first verify that the triple ${\T}_{\h}(M)=(\V_\ZZ,\W_*,\F^*)$ associated to $M$ is a torsion-free, graded-polarizable variation of mixed Hodge structures over $S^\an$ of type (0,0), (-1,0),(0,-1) or
	(-1,-1). Denote by $f_T: T \rightarrow S$ and $f_A: A \rightarrow S$  the structural morphisms of the torus $T$ and of the abelian scheme $A$ underlying the 1-motive $M$. Since  ${\R}_1 (f_T)_*^{\an}{\ZZ} \otimes_\ZZ \cO_{S^\an} \cong \li T  $, we have the equality $\W_{-2} (\V_\ZZ) \cap \F^0=0.$ Hence  $ (\W_{-2} (\V_\ZZ), \W_*, \F^*)$ is a torsion-free variation of mixed Hodge structures of type (-1,-1). Then we have an isomorphism between $\W_{-1} (\V_\ZZ) \cap \F^0$ and 
	$\ker \big( {\R}_1 (f_A)_*^{\an}{\ZZ} \otimes_\ZZ \cO_{S^\an} \rightarrow \li A \big)$, and so $\F^*$ induces on $\Gr_{-1}^\W(\V_\ZZ) \otimes_\ZZ \cO_{S^\an} \cong {\R}_1 (f_A)_*^{\an}{\ZZ} \otimes_\ZZ \cO_{S^\an} $ the Hodge filtration of ${\R}_1 (f_A)_*^{\an}{\ZZ} \otimes_\ZZ \cO_{S^\an} $. Therefore  $(\Gr_{-1}^\W(\V_\ZZ),  \W_*, \F^*)$  is a torsion-free, graded-polarizable variation of mixed Hodge structures of type (-1,0) and (0,-1). Moreover the short exact sequence $0 \rightarrow \W_{-1} (\V_\ZZ) \cap \F^0 \rightarrow \F^0 \rightarrow \Gr_0^\W(\V_\ZZ) \rightarrow 0$ implies that $(\Gr_{0}^\W(\V_\ZZ),  \W_*, \F^*)$ is a torsion-free variation of mixed Hodge structures of type (0,0). 
	Finally ${\T}_{\h}(M)=(\V_\ZZ,\W_*,\F^*)$ is admissible because of \cite[Lemma 5]{A}. This construction, which assigns to each 1-motive $M$ the variation of mixed Hodge structures ${\T}_{\h}(M)=(\V_\ZZ,\W_*,\F^*),$ is functorial.

	(2) Now we prove that the functor $M \mapsto {\T}_{\h}(M)$ is essentially surjective. Let $\V=(\V_\ZZ,\W_*,\F^*)$ be a torsion-free, graded-polarizable, admissible variation of mixed Hodge structures over $S^\an$ of type (0,0), (-1,0),(0,-1), (-1,-1) and denote by $M^\an=[u^\an: X^\an \to  G^\an]$ the corresponding analytic 1-motive. By \cite[Rappel (4.4.3)]{D0} the torsion-free, graded-polarizable variation of mixed Hodge structures of type (-1,0) and (0,-1) underlying $\V$ comes from a polarizable abelian scheme $A$ over $S$. Let $T$ be the $S$-torus whose character group $X^*(T)$ is the Cartier dual of $\Gr_{-2}^\W(\V_\ZZ)$ and let $X$ be the lattice $\Gr_{0}^\W(\V_\ZZ).$
 Since the category of extensions is additive, we may assume that $T=\GG_m$ is 1-dimensional and that $X= \ZZ $ has rank 1. 
	
\par\noindent	According to \cite[Theorem 3.2.1]{A25}, the dual of the torsion-free, graded-polarizable, admissible variation of mixed Hodge structures $\Gr_{-1}^\W(\V_\ZZ)$ of type (-1,0), (0,-1) and (-1,-1) underlying $\V$ has an algebraic origin given by a 1-motive $[\hat{v}: \ZZ \to  \widehat{A}],$ where $\widehat{A}$ is the Cartier dual of $A.$ The Cartier dual of the 1-motive 
	$[\hat{v}: \ZZ \to  \widehat{A}]$ is the extension $G$ of $A$ by $T$, whose complex analytic group scheme $G^\an$ is $\W_{-1}(\V_\ZZ)  \setminus \W_{-1}(\V_\ZZ \otimes_\ZZ \cO_{S^\an}) / \F^0 \cap \W_{-1}(\V_\ZZ \otimes_\ZZ \cO_{S^\an})$). 
	
	\par\noindent In order to show that also the map $u^\an: \ZZ \to  G^\an$ has an algebraic origin given by a map $u: \ZZ \to G$, we apply Theorem \ref{Teo:AndreG} (1).

	(3) Finally we check that the functor $M \mapsto {\T}_{\h}(M)$ is fully faithful. Let $M_i=(X_i,A_i,T_i,G_i,u_i)$ be two 1-motives defined over $S$ for $i=1,2.$ We will show that the $S$-sheaf $\uHom(M_1,M_2)$ is unramified over $S,$ which implies that its algebraic and holomorphic sections coincide.
	
	 We proceed by dévissage. We first prove that the 
	 $S$-sheaf $\uHom(T_1,T_2)$ is unramified over $S.$ It is enough to show that if $S = \spec (R)$ for a local Artin ring $R$ with residue field $k$ and if $f: T_1 \rightarrow T_2$ is an $S$-morphism whose restriction over ${\spec (k)}$ is trivial, then $f$ is trivial. Consider the restriction $f_{i}:T_1[q^i] \rightarrow T_2[q^i]$ of $f$ to the points of order $q^i$ with
	$q$ an integer bigger than 1 which is coprime with the characteristic of $k$, and $i>0$. For each $i$, $f_{i}$ is a morphism of finite and \'etale $S$-schemes which is trivial over ${\spec (k)}$, and so it is trivial.
	Since the family $(T_1[q^i])_i$ is schematically dense (see \cite[Expos\'e IX Theorem 4.7 and Remark 4.10]{SGA3}), we can conclude.
	
	The long exact sequence 
	\[0 \to \uHom(G_1,T_2) \to \uHom(G_1,G_2) \stackrel{\phi}{\to} \uHom(G_1,A_2) \to \uExt^1(G_1,T_2) \to  \cdots\]
	implies that $\uHom(G_1,G_2)$ is a $\uHom(G_1,T_2)$-torsor over $\im(\phi) \subseteq \uHom(G_1,A_2) .$
	By \cite[Proposition 1.4.1]{B09-2} the $S$-sheaf $\uHom(G_1,T_2)$ is a sub-sheaf of the unramified  $S$-sheaf $\uHom(T_1,T_2),$ and so it is unramified over $S$ too (the notion of unramified depends on stalks, see \cite[Theorem (17.4.1)]{EGAIV4}). The $\uHom(G_1,T_2)$-torsor $\uHom(G_1,G_2)$ is then unramified over its basis $\im(\phi).$
 According to \cite[Proposition 1.4.1]{B09-2} the $S$-sheaf $\uHom(G_1,A_2)$ is isomorphic to the unramified $S$-sheaf $\uHom(A_1,A_2)$ which implies that its $S$-sub-sheaf $\im(\phi)$ is unramified over $S$ too. By \cite[Theorem (17.3.3)]{EGAIV4}) it follows that $\uHom(G_1,G_2)$ is unramified over $S.$\footnote{We could have considered the long exact sequence 
 	$0 \to \uHom(A_1,G_2) \to \uHom(G_1,G_2) \to  \uHom(T_1,G_2) \to \uExt^1(A_1,G_2) \to  \cdots$
 	and the facts that $\uHom(A_1,G_2)$ is a sub-sheaf of the unramified sheaf $\uHom(A_1,A_2)$ and $ \uHom(T_1,G_2)$ is isomorphic to the unramified $S$-sheaf $\uHom(T_1,T_2)$ (see \cite[Proposition 1.4.1]{B09-2}).}
 
 By \cite[Proposition 1.1.3]{B09-2} the $S$-sheaf $\uHom(G_1,X_2[1])$ is trivial, which implies that the $S$-sheaves $\uHom(G_1,G_2)$ and $\uHom(G_1,M_2)$ are isomorphic. Therefore $\uHom(G_1,M_2)$ is unramified over $S.$
 The triviality of the $S$-sheaf $\uHom(X_1[1],G_2)$ implies that the $S$-sheaf $\uHom(X_1[1],M_2)$ is a sub-sheaf of the étale $S$-sheaf $\uHom(X_1,X_2),$ hence it is étale over $S$ too.\footnote{In a similar way we show that $\uHom(M_1,G_2)$ is a sub-sheaf of the unramified $S$-sheaf $\uHom(G_1,G_2)$ and so it is unramified too, and $\uHom(M_1,X_2[1])$  is isomorphic to the étale $S$-sheaf $\uHom(X_1,X_2),$ hence it is étale too.}
 
 	The long exact sequence 
 \[0 \to \uHom(X_1[1],M_2) \to \uHom(M_1,M_2) \stackrel{\varphi}{\to}  \uHom(G_1,M_2) \to \uExt^1(X_1[1],M_2) \to  \cdots\]
 implies that $\uHom(M_1,M_2)$ is a $\uHom(X_1[1],M_2)$-torsor over $\im (\varphi) \subseteq  \uHom(G_1,M_2),$ hence
  it is unramified (even étale) over its basis $\im (\varphi).$
 Since  $\uHom(G_1,M_2)$ is unramified over $S,$ its $S$-sub-sheaf $\im (\varphi)$ is unramified over $S$ too. Therefore by \cite[Theorem (17.3.3)]{EGAIV4}) $\uHom(M_1,M_2)$ is unramified over $S.$\footnote{We could have considered the long exact sequence 
 	$0 \to \uHom(M_1,G_2) \to \uHom(M_1,M_2) \to  \uHom(M_1,X_2[2]) \to \uExt^1(M_1,G_2) \to  \cdots$
 	and the facts that $\uHom(M_1,G_2)$ is a sub sheaf of the unramified sheaf $\uHom(G_1,G_2)$ and $\uHom(M_1,X_2[2])$ is isomorphic to the étale $S$-sheaf $\uHom(X_1,X_2)$.}
\end{proof}

 By Theorem \ref{VMHS1mot} the enriched Hodge realization functor
\[
\T_{\h} : 1-\mathrm{Mot}(S) \longrightarrow \mathrm{VMHS}^{\mathrm{ad}}(S) \subset \mathrm{VMHS}(S),
\qquad
M \longmapsto T_\h(M)
\]
is fully faithful, that is
\[
\operatorname{Hom}_{1\text{-}\mathrm{Mot}(S)}(M_1,M_2)
\cong
\operatorname{Hom}_{\mathrm{VMHS}^{\mathrm{ad}}(S)}
\bigl(\T_\h(M_1),\T_\h(M_2)\bigr).
\]

\begin{definition}\label{def:MS(S)}
	Let $S$ be a connected scheme smooth and of finite type over a field $k$ of characteristic 0. The Tannakian category $\vmhs^\ad(S)$ of admissible variations of mixed Hodge structures over $S$ is the inductive limit 
	\[\vmhs^\ad(S) := \varinjlim \vmhs^\ad(S') \] 
	of the Tannakian categories $\vmhs^\ad(S')$,
	indexed by all connected schemes $S',$ which are smooth and of finite type over a field $k'$ of finite type over $\QQ$, endowed with an arrow $k' \to k$ and with an isomorphism $S \cong S' \times_{\spec k'} \spec k$.
	In order to be sure that $k$ embeds into $\CC$, we assume that the cardinality of $k'$ is smaller or equal to continuum.
\end{definition}

\pn Via the inductive limit we can extend the notion of enriched Hodge realization functor to 1-motives over a $k$-scheme $S$ with $k$ a field of characteristic 0 and thereby introduce the following notion

\begin{definition}
	Let $M$ be a 1-motive defined over a connected scheme $S$ which is smooth and of finite type over a field $k$ of characteristic 0. The \emph{global Mumford-Tate group} of $M,$ denoted $ \mathrm{G}_{\vmhs^\ad (S)}(M),$ is the fundamental group of the Tannakian subcategory $<{\T}_{\h}(M)>^\otimes$ of $\vmhs^\ad (S)$ generated by the enriched Hodge realization ${\T}_{\h}(M)=(\V_{\ZZ},\W_*,\F^*)$ of $M$.
\end{definition}

Fix a point $s$ of $S$. The category of local systems $\locsys$ 
forms a Tannakian category with fibre functor given by \emph{taking the fibre at the base point $s$}. Its fundamental group is
the topological fundamental group $\pi_1(S,s).$ 
The restriction of the functor \emph{forget the Hodge structure} 
$\vmhs^\ad (S) \rightarrow \locsys $ 
to the Tannakian subcategories $<{\T}_{\h}(M)>^\otimes$ generated by the enriched Hodge realization of a 1-motive $M$ over $S$ induces an essentially surjective functor between the corresponding Tannakian categories 
$<{\T}_{\h}(M)>^\otimes \to < \V_\ZZ>^\otimes .$
By Tannakian duality, we get an inclusion of algebraic groups 
\[G_{\mathrm{mon}}(\V_{\ZZ, s}) \longhookrightarrow \mathrm{G}_{\vmhs^\ad (S)}(M),\]
 where 
 $G_{\mathrm{mon}}(\V_{\ZZ, s})$ is the fundamental group of the Tannakian subcategory $<\V_\ZZ >^\otimes$ of $\locsys$ generated by $\V_\ZZ,$ that is the algebraic monodromy group of the local system $\V_\ZZ.$

\begin{proposition}\label{MT}
Let $M$ be a 1-motive defined over a connected scheme $S$ which is smooth and of finite type over a field $k$ of characteristic 0.
 The neutral connected component of the global Mumford-Tate group $\mathrm{G}_{\vmhs^\ad (S)}(M)$ of $M$ identifies with the Mumford-Tate group of the fiber $M_{s_0}$ at the generic point $s_0 $ of $ S$:
	\[	\mathrm{G}_{\vmhs^\ad (S)}(M)^\circ = \mathrm{MT}(M_{s_0}).	\]
\end{proposition}

\begin{proof} Let $s$ be a point of $S$.
	A tensor in a tensor construction on the mixed Hodge structure $V_s$ is fixed by the algebraic monodromy group $G_{\mathrm{mon}}(\V_{\ZZ, s})$
	if and only if it is invariant under monodromy.
	By the theorem of the fixed part, such tensors arise from constant subobjects of the admissible variation $\V_\ZZ$ of mixed Hodge structures, hence are independent of the point $s$ of $S$.

	The algebraic monodromy group is contained
	in the Mumford-Tate group $\MT(M_{s_0})$, hence the Hodge tensors of the mixed Hodge structure
	$\V_{\ZZ, s_0}$ are invariant under monodromy. In particular they coincide with the global (i.e. constant) tensors. Therefore,
	the Mumford-Tate group $\mathrm{MT}(M_{s_0})$ is the stabilizer of the Hodge tensors,
	which coincide with the global tensor invariants defining  $\mathrm{G}_{\vmhs^\ad (S)}(M)^\circ .$ This proves the proposition.
\end{proof}

\begin{remark}\label{remMT}
	Since the Mumford-Tate group $\mathrm{MT}(M_{s_0})$ is connected, while the global Mumford-Tate group $	\mathrm{G}_{\vmhs^\ad (S)}(M)$
	may have non-trivial discrete components arising from monodromy, the comparison
	holds at the level of the neutral connected component.
\end{remark}


\bibliographystyle{plain}

\end{document}